\documentclass[11pt]{amsart}
\usepackage[all]{xy}
\usepackage{array}
\usepackage{amsthm}
\usepackage{amsmath}
\usepackage{verbatim}
\theoremstyle{definition}
\theoremstyle{plain}
\newtheorem{lemma}{Lemma}[section]
\newtheorem*{lemma*}{Lemma}
\newtheorem{theorem}{Theorem}[section]
\numberwithin{equation}{section}

\newcommand{\rightlimit}[1]{\mathop{\lim\limits_{\longrightarrow}}\limits%
                   _{\raise3pt\hbox{$\scriptstyle #1$}}}

\title{The action of the Cremona group on the non-commutative ring.}
\begin{document}
\author{Alexandr Usnich}
\maketitle
\begin{abstract}
 The Cremona group acts on the field of two independent commutative variables over complex numbers.
 We provide a non-commutative ring that is an analog of non-commutative field of two independent variables
 and prove that the Cremona group embeds in the group of outer automorphisms of this ring. First proof of this
 result is technical, the second one is conceptual and proceeds through a construction of a birational invariant
 of an algebraic variety from it's bounded derived categories of coherent
 sheaves.
\end{abstract}

\section{Introduction}

    The Cremona group is the group of automorphisms of the field
$K=\mathbf{C}(x,y)$ of two independent commuting variables.
Alternatively it is the group of birational automorphisms of
$\mathbf{CP}^2$ or any other rational surface, because $K$ is it's
field of rational functions.\\
    We construct a non-commutative associative algebra $A$, which is a
consecutive localization of a free non-commutative algebra of two
generators at the set of elements outside a commutator ideal. The
main result is that the Cremona group acts on this
non-commutative algebra by outer automorphisms.\\
First proof of this statement is a direct computation: we provide
non-commutative analogs of generators of Cremona group and verify
that the relations known in commutative setting also hold in a
non-commutative up to inner automorphisms.\\
    Then we prove the same fact more abstractly, namely for a smooth proper algebraic variety
$X$ we construct a triangulated category $\widetilde{C}(X)$, which
is a birational invariant of $X$. By this we mean that if $X$ and
$Y$ are birationally equivalent, then $\widetilde{C}(X)$ and
$\widetilde{C}(Y)$ are equivalent as triangulated categories. The
category $\widetilde{C}(X)$ is a quotient of derived category of
coherent sheaves that are left orthogonal to the structure sheaf
of $O_X$ by the full subcategory of complexes of sheaves with the
support of codimension at least $1$. Then we study the structure
of $\widetilde{C}(\mathbf{CP^2})$ and prove that it is generated
by one object $P$ which is the image of $O(1)$. We also check that
this object is preserved under equivalences, so the Cremona group
will act by outer automorphisms on the ring of endomorphisms of
this object. Then we compute that
$RHom^0_{\widetilde{C}(\mathbf{CP}^2)}(P,P)=A$, so we get an
action of the Cremona group on the same non-commutative algebra as
previous. Then we verify that two actions coincide.\\
    I would like to express the gratitude to my advisor Maxim Kontsevich
for his inspiration, ideas and a lot of time that he generously
consecrated to me. Also I should thank Takuro Mochizuki for his
very helpful questions about those results.\\

\section{Construction of the noncommutative algebra $A$}\label{construction_A}

    In what follows we will construct a non-commutative algebra $A$ together with algebra
homomorphism $\phi:A\rightarrow K$ such that:\\
1. $R=\mathbf{C}<x,y>\subset A$, $R$ is an algebra of polynomials in
two non-commuting variables.\\
2. $I=A[A,A]A$ is the kernel of $\phi$\\
3. All the elements in $A\setminus I$ are invertible\\
\\
 Let us construct such an algebra by induction. Put $A_0:=R=\mathbf{C}<x,y>$,
$I_0:=R[R,R]R=R(xy-yx)R$(just notice that for associative
noncommutative algebras $B$ the following holds
$B[B,B]B=[B,B]B=B[B,B]$ and we call this submodule a commutator of
$B$ ). We successively construct $A_{i+1}$ as a localization of
$A_i$ at $A_i\setminus I_i$. $I_i$ is just a commutator of $A_i$,
denote $S_i=A_i\setminus I_i$. First let us explain what do we
mean by localization: consider $B_i:=A_i<u_s|s\in S_i>$ and let
$P_i$ be a bi-ideal in $B_i$ generated by $u_s s-1$ and $s u_s-1$.
We put $A_{i+1}:=B_i/P_i$. We also have induced maps
$A_i\rightarrow A_{i+1}$, so we may take an inductive limit
$A:=\underrightarrow{lim}A_i$.\\
    We verify the following:
\begin{lemma} $(1)$. $\phi_i:A_i\rightarrow K$ is well defined and
consistent with the maps $A_i\rightarrow A_{i+1}$.\\
$(2)$. $ker(\phi_i)=I_i$.\\
\end{lemma}
\textbf{Corollary}.  $\phi:A\rightarrow K$ is well defined and
$ker(\phi)=I$.\\ We can define $\phi$ on the inductive limit of
$A_k$'s which is $A$ and the statement $ker(\phi)=I$ follows from
the corresponding statement about $\phi_k$.
\begin{proof}  For $i=0$ we define an obvious map $\phi_0:R\rightarrow \mathbf{C}[x,y]\subset K$. It has
the commutator $I_0$ as the kernel.\\
    By induction we suppose that
$ker(\phi_i)=I_i$ for $i=k$. Then $\phi_k$ maps $S_k$ to non-zero
thus invertible elements of $K$, so we may extend the map
$\phi_k(u_s):=\phi_k(s)^{-1}$ to $B_k$ and it is easily seen that
$P_k$ maps to $0$, so we may define
$\phi_{k+1}$ in a consistent way on $A_{k+1}=B_k/P_k$.\\
    So $S_k$ may be characterized as the set of elements that maps by
$\phi_k$ not to zero, so it is a multiplicative system. We may
verify the following equality in $A_{k+1}$:\\
 $u_s u_t=u_{ts}$, because $ts(u_{ts}-u_s
u_t)=(tsu_{ts}-1)+t(su_s-1)u_t+(tu_t-1)\in P_k$. From the other
side $(u_{ts}ts-1)(u_{ts}-u_s u_t)\in P_k$. It follows that
$(u_{ts}-u_s u_t)\in P_k$.\\
    Let us prove now that $ker(\phi_{k+1})=I_{k+1}$ by induction on $k$.
Commutator  of $A_{k+1}$ is contained in the kernel of
$\phi_{k+1}$, because $K$ is commutative, so we have to prove
another inclusion. For any element $b\in B_k$ there exist $a,s\in
A_k$ such that $b-au_s\in B_k[B_k,B_k]+P_k$, because
\[a_1u_{s_1}a_2u_{s_2}..a_n-a_1a_2...a_nu_{s_1}u_{s_2}..u_{s_{n-1}}\in
B_k[B_k,B_k]\]
\[au_{s_1}u_{s_2}..u_{s_k}-au_{s_k...s_2s_1}\in P_k\]
\[au_s+bu_p-(ap+bs)u_{ps}\in B_k[B_k,B_k]+P_k\]
In particular for any $b'\in A_{k+1}=B_k/P_k$ there exist
$a',s'\in A_k$, such that $b'-a'u_s'\in
A_{k+1}[A_{k+1},A_{k+1}]=I_{k+1}$. So if $\phi_{k+1}(b')=0$ then
$\phi_{k+1}(a'u_{s'})=\phi_k(a')\phi_k(s')^{-1}=0$, so
$\phi_k(a')=0$ and by induction hypothesis $a'\in
A_k[A_k,A_k]=I_k$. This proves that $b'\in I_{k+1}$.
\end{proof}

    To define a morphism from an algebra $A$ to any algebra $B$ we
should prescribe images of $x$,$y$. Then a map from $A_0$ to $B$
will be defined. If the image of elements not in commutator of
$A_k$ is invertible in $B$, then we can extend it to the map from
$A_{k+1}$ to $B$.

\begin{lemma}
    Suppose $a,b\in A$ such that the field morphism $i:K\rightarrow K$
given by $i:x\mapsto \phi(a)$, $i:y\mapsto \phi(b)$ is injective.
Then there is a well defined algebra homomorphism $f:A\rightarrow A$
with $f(x)=a$, $f(y)=b$.
\end{lemma}
\begin{proof}
    Given $a,b\in A$ we uniquely define $f_0:A_0\rightarrow A$ by
prescribing $f_0(x)=a$, $f_0(y)=b$. Now by induction suppose that
the map $f_k:A_k\rightarrow A$ is well defined and $f_k(x)=a$,
$f_k(y)=b$. Now we use $\phi\circ f_k=i\circ\phi_k$. It follows that
$\phi(f(S_k))$ is non-zero, because $\phi_k(S_k)$ is non-zero and
$i$ is injective. So $f_k(S_k)$ is invertible and we can uniquely
extend $f_k$ to $f_{k+1}:A_{k+1}\rightarrow A$.
\end{proof}

    To see that $A$ is indeed non-commutative and has non-trivial
elements(it's not obvious from the construction that $A$ doesn't
coincide with $K$) let us construct algebra morphisms from $A$ to
the following algebra $M=M_{k\times k}(K[[\epsilon]])$ - algebra
of matrices over $K[[\epsilon]]$ series with coefficients in $K$.
Let us choose two matrices of the form $X=xId+\epsilon X_1$,
$Y=yId+\epsilon Y_1$, where $X_1,Y_1$ are some matrices in $M$. We
can invert element of $M$ if and only if we can invert it's
reduction modulo $\epsilon$. This justifies that the map
$R\rightarrow M$ given by $x\mapsto X$, $y\mapsto Y$ may be
extended to the morphism $A\rightarrow M$.\\
    For example we see that $XY-YX=\epsilon^2(X_1Y_1-Y_1X_1)$ may be
non-zero, implying that $xy-yx$ is nonzero in $A$, so $A$ is indeed non-commutative.\\
\begin{lemma}
    The natural map $i:A_0=\mathbf{C}<x,y>\rightarrow A$ is an injection.
\end{lemma}
\begin{proof}
    Let us consider an algebra $M=M_{k\times k}(K[[\epsilon]])$ of
matrices of big enough size $k$ with entries in the ring of formal
power series $K[[\epsilon]]$. For any pair of complex matrices
$S,T\in M_{k\times k}(\mathbf{C})$ it exists a ring homomorphism
$f:A\rightarrow M$ given by $x\mapsto xId+\epsilon S$, $y\mapsto
yId+\epsilon T$. So if some non-commutative polynomial $a\in A_0$
was in the kernel of the map $i:A_0\rightarrow A$, then $f\circ
i(a)=0$ and it's a polynomial in $\epsilon$ with coefficients in
matrices, in particular for the highest power of $\epsilon$ we
would have $a(S,T)=0$. So we would have that some non-commutative
polynomial is identically $0$ if we put as arguments any complex
matrices of any size. This is false, so we proved the lemma.
\end{proof}

\section{Direct extension of the Cremona group to non-commutative variables}

    Suppose $a=\begin{pmatrix}P(x)& Q(x)\\ R(x)& S(x)\end{pmatrix}\in
PGL(2,k(x))$. It gives an automorphism of $K$ over $\mathbf{C}$:

$$a:(x,y)\mapsto (x, \frac{yP(x)+Q(x)}{yR(x)+S(x)})$$

    Define also a map $\tau:(x,y)\mapsto (y,x)$.\\

    Let us summarize the results of \cite{Iskovskih85}:\\
    The Cremona group $Cr$ is generated by $\tau$ and $PGL(2,k(x))$ with the described
action.\\
    Consider an exact sequence of groups $1\rightarrow
PGL(2,k(x))\rightarrow B\rightarrow PGL(2,k)\rightarrow 1$, where
$B$ preserves the subfield $k(x)\subset k(x,y)$. $D=PGL(2,k)\times
PGL(2,k)=D_1\times D_2\subset B$ preserves both $k(x)$ and $k(y)$,
where the action of $M=\begin{pmatrix}a&&b\\c&&d\end{pmatrix}\in
D_1$ is given by $(x,y)\mapsto (\frac{ax+b}{cx+d},y)$. And the
action of $N=\begin{pmatrix}a&&b\\c&&d\end{pmatrix}\in D_2$ is
 given by $(x,y)\mapsto(x,\frac{ay+b}{cy+d})$.\\

    The Cremona group is generated by $\tau, B$ and by
\cite{Iskovskih85} all the relations are the consequences of the following:\\
1. $\tau^2=1$\\
2. $\tau D \tau=D$\\
3. $(\tau e)^3=a$\\
in the last relation $e:(x,y)\mapsto(x,\frac{x}{y})\in B$ and
$a:(x,y)\mapsto(\frac{1}{x},\frac{1}{y})\in D$\\

    Now let us reinterpret these relations: we denote
$D_1=D_2=PGL(2,k)$ and $G=PGL(2,k(x))$. Field inclusion $k\subset
k(x)$ induces group inclusion $D_2=G(k)\subset G=G(k(x))$. $D_1$
acts on $k(x)$, so we have also the action of $D_1$ on
coefficients of $PGL(2,k(x))=G$ which leaves $D_2$ stable.\\
    Consider the group generated by $G$ and
$\mathbf{Z}/2=<1,\tau>$, on which we put the following relations:\\
$1$. $D_1=\tau D_2 \tau$ normalizes $G$ and the induced action
of $D_1$ on $G$ is induced from the action of $PGL(2,k)=D_1$ on $k(x)$.\\
$2$. $(\tau e)^3=a$, where
$e=\begin{pmatrix}0&&x\\1&&0\end{pmatrix}\in G$ and
$a=\begin{pmatrix}0&&1\\1&&0\end{pmatrix}\times\begin{pmatrix}0&&1\\1&&0\end{pmatrix}\in
D_1\times D_2$.\\
    The group $B$ from previous notations is included in the exact
sequence: $1\rightarrow G\rightarrow B\rightarrow D_1\rightarrow
1$. So the the group that we described is the Cremona group, as
this set of relations is equivalent to the previous one.\\

    Let us now introduce an action on the previously constructed ring
$A$. $a\in GL(2,k(x))$ is a matrix $\begin{pmatrix}P(x) && Q(x)\\
R(x) && S(x)\end{pmatrix}$. It gives automorphisms of $A$:
$$t_a:(x,y)\mapsto (x, (yR(x)+S(x))^{-1}(yP(x)+Q(x)))$$\label{explicit_aut}
$$p_a:(x,y)\mapsto (x, (P(x)y+Q(x))(R(x)y+S(x))^{-1})$$
Before proceeding further, let us show how those automorphisms are
related: the dual algebra for $R=\mathbf{C}<x,y>$ would be
$R^{op}=\mathbf{C}<x,y>$. They are isomorphic as algebras:
\begin{align*}
&\rho:&R&\longrightarrow & R^{op}\\
&     &(x,y)&\longmapsto & (x,y)
\end{align*}
    In particular $\rho(xy)=yx$ and $\rho(x^2y^3x^7)=x^7y^3x^2$.\\
We can extend $\rho$ as an isomorphism between $A$ and $A^{op}$,
so we would have $\rho(P^{-1}Q)=\rho(Q)\rho(P)^{-1}$. Now we
observe that formulas for $t_a$ and $p_a$ are dual:
$$\rho\circ t_a=p_a\circ\rho$$
    So the result we would prove for $t_a$ would also be true for
$p_a$.\\
    Define a map $\tau:(x,y)\mapsto (y,x)$, and consider a group
$Cr^{nc}$ generated by $\tau$ and $t_a$ for all $a\in GL(2,k(x))$.
Commutator $A[A,A]A$ is preserved by any automorphism, so
$Cr^{nc}$ also acts on $A/A[A,A]A=\mathbf{C}(x,y)$ i.e. maps to
commutative Cremona group $Cr$. Let us denote $Cr^{in}$ the
subgroup of $Cr^{nc}$ generated by inner automorphisms, i.e. of
the type $(x,y)\mapsto(rxr^{-1},ryr^{-1})$ for some invertible
$r\in A^*$. We want to prove that $Cr^{in}$ is the kernel of the
map $Cr^{nc}\rightarrow Cr$.\\

\begin{theorem} The kernel of the map $Cr^{nc}\rightarrow Cr$
consists of inner automorphisms of $A$.
\end{theorem}
\begin{proof}
    For this we need to check that when we lift a relation for
generators of $Cr$ to $Cr^{nc}$ it belongs to $Cr^{in}$.

\begin{lemma}  $t_at_b=t_{ab}$ for any $a,b\in GL(2,k(x))$\end{lemma}
\begin{proof}
    Let $b=\begin{pmatrix} P(x) && Q(x)\\R(x) && S(x)\end{pmatrix}$,
$a=\begin{pmatrix} P_1(x) && Q_1(x)\\R_1(x) &&
S_1(x)\end{pmatrix}$. The action of $a$ on $A$ is given by
$$t_b:(x,y)\mapsto(x,(yR(x)+S(x))^{-1}(yP(x)+Q(x)))=:(x,y_1)$$
$$t_a:(x,y_1)\mapsto(x,(y_1R_1(x)+S_1(x))^{-1}(y_1P_1(x)+Q_1(x)))=:(x,y_2)$$
$$y_1R_1(x)+S_1(x)=(yR+S)^{-1}(y(PR_1+RS_1)+(QR_1+SS_1))$$
$$y_1P_1(x)+Q_1(x)=(yR+S)^{-1}(y(PP_1+RQ_1)+(QP_1+SQ_1))$$
$$t_at_b:(x,y)\mapsto(x,(y(PR_1+RS_1)+(QR_1+SS_1))^{-1}(y(PP_1+RQ_1)+(QP_1+SQ_1)))$$
So $t_at_b=t_{ab}$, because
\[ab=\begin{pmatrix}P_1(x)P(x)+Q_1(x)R(x) &&
P_1(x)Q(x)+Q_1(x)S(x)\\R_1(x)P(x)+S_1(x)R(x) &&
R_1(x)Q(x)+S_1(x)S(x)\end{pmatrix}\]
\end{proof}

    When $GL(2,k(x))\subset Cr^{nc}$ maps to $Cr$ it has a center in
the kernel, so suppose $d=\begin{pmatrix}d(x)&& 0\\ 0 &&
d(x)\end{pmatrix}$ is a diagonal matrix, then
$t_{d}:(x,y)\mapsto(x, d(x)^{-1}yd(x))$ is a conjugation by
$d(x)$, so $t_d\in Cr^{in}$.\\

    Let us verify a first relation: suppose
$m=\begin{pmatrix}a&&b\\c&&d\end{pmatrix}\in D_2=GL(2,k)$. Then
$m_1=\tau t_m\tau:(x,y)\mapsto((cx+d)^{-1}(ax+b),y)$. From this
formula we easily see that $m_1$ normalizes a subgroup
$GL(2,k(x))\subset Cr^{nc}$ and it's action is induced from
it's action on $k(x)$, so the first relation is verified.\\

    Let us now verify the second relation: $(\tau e)^3=a$. Remind that
$$\tau:(x,y)\mapsto(y,x)$$
$$e:(x,y)\mapsto(x,y^{-1}x)$$
$$a:(x,y)\mapsto(x^{-1},y^{-1})$$
    Then we calculate
$$\tau e:(x,y)\mapsto(y^{-1}x,x)$$
$$(\tau e)^2:(x,y)\mapsto(x^{-1}y^{-1}x,y^{-1}x)$$
$$a^{-1}(\tau e)^3:(x,y)\mapsto(x^{-1}yxy^{-1}x,x^{-1}yx)$$
    The last automorphism is a conjugation by $x^{-1}y$, so
$a^{-1}(\tau e^3)\in Cr^{in}$.
\end{proof}

\section{The proof using derived categories}

    In this section we would construct a triangulated category $\widetilde{C}$
with a specific isomorhism class of objects $P$, and it would be a
birational invariant of rational surfaces. As a consequence, the
Cremona group will act on this category by equivalences and will
preserve $P$, so it will act by outer automorphisms on the graded
ring $H^*(End_C(P))$, in particular it's degree $0$ part coincides
with the non-commutative algebra $A$ considered previously. We
would then verify that the action of $Cr$ on $A$ coincides with
the one given by explicit formulas.\\

\subsection{Notations}
    First let us make some conventional notations:\\
$D$ - bounded derived category of coherent sheaves on $\mathbf{CP}^2$\\
$\pi:X\rightarrow\mathbf{CP}^2$ -  sequence of blow-ups of $\mathbf{CP}^2$ or equivalently a regular map invertible outside a finite set of points.\\
$D(X)$ - bounded derived category of coherent sheaves on $X$. By
omitting $X$ we would mean the corresponding category for
$\mathbf{CP}^2$.\\
$O,O(1),O(2)$ - line bundles
$O_{\mathbf{CP}^2},O_{\mathbf{CP}^2}(1),O_{\mathbf{CP}^2}(2)$, we
would also see them as objects of $D$ placed at degree $0$.
Depending on the context we would denote by the same symbols the
objects $L\pi^*O,L\pi^*O(1),L\pi^*O(2)\in D(X)$ for the functor
$L\pi^*:D\rightarrow D(X)$.\\
$E$ - exceptional curve of a blow-up.\\
$\widetilde{D}(X)$ - full subcategory of $D(X)$ consisting o
complexes of sheaves left orthogonal to $O_X$.\\
$D^1(X)$ - full subcategory of $D(X)$ of complexes of sheaves with
the dimension of support $\le 1$.\\
$\widetilde{D}^1(X)$ - full subcategory of $D(X)$ of objects which
are both in $\widetilde{D}(X)$ and in $D^1(X)$.\\
$RHom(F,G)$ would usually be a derived functor in category $D(X)$.
If both $F,G$ belong to $D$, we consider $RHom$ in $D$, because
$L\pi^*$ is a fully faithful functor.\\
$\widetilde{C}(X)=\widetilde{D}(X)/\widetilde{D}^1(X)$ - quotient
triangulated category.\\
$P\in\widetilde{C}(X)$ - isomorphism class of object $O(1)$.\\
$a,b,e\in Hom(O,O(1))$ - sections corresponding to $X,Y,Z$.\\
$R=\mathbf{C}<x,y>$ - ring of polynomials in non-commutative
variables.\\
$A$ - non-commutative algebra, constructed in previous sections.\\
$K=\mathbf{C}(x,y)$\\
$A^{\bullet}[1]$ a shift of a complex $A^{\bullet}$ to the left\\

\subsection{Strategy of the proof}

First we prove that the quotient category
$\widetilde{C}=\widetilde{D}/\widetilde{D}^1$ is split generated
by one object $P$ and $RHom_{\widetilde{C}}(P,P)\cong A^{\bullet}$
is some dg-algebra. We prove that this dg-algebra is concentrated
in non-positive degrees and $H^0(A^{\bullet})=A$ is a
non-commutative algebra that we constructed earlier.\\
Next we prove that the functor $L\pi^*:D\rightarrow D(X)$ induces an
equivalence of categories $\widetilde{C}$ and $\widetilde{C}(X)$ for
any sequence of blow-ups $\pi:X\rightarrow\mathbf{CP}^2$. From this
we construct an action of the Cremona group on the category
$\widetilde{C}$ by equivalences of category.\\
Then we verify, that this action preserves the isomorphism class
of object $P$, so the Cremona group would act by outer
automorphisms on $A$.\\
    At last we compute this action on elements $x,y\in A$ and verify
that it coincides with our explicit formulas.\\

    To start with, we show how this construction works for the category
$D$ without taking left orthogonal to the structure sheaf, so we
get an action of the Cremona group on the category $C=D/D^1$. In
this case we just recover the action of Cremona group on the field
$K=\mathbf{C}(\mathbf{CP}^2)$. The functor
$\widetilde{C}\rightarrow C$ induced from inclusion
$\widetilde{D}\subset D$ corresponds to the
"commutativization" morphism $comm:A\rightarrow K$.\\

\subsection{Action of the Cremona group on $K$ through derived
categories}

$\mathbf{C}(X)$ the field of rational functions on $X$.
$D^b(\mathbf{C}(X)-mod)$ - bounded derived category of finitely
generated modules over it.
\begin{lemma} The functor of restriction to a generic point induces an equivalence of triangulated categories
$C(X)=D(X)/D^1(X)$ and $D^b(\mathbf{C}(X)-mod)$.\end{lemma}
\begin{proof}
    Let $\eta(F)$ be a restriction of a coherent sheaf $F$ from $X$ to
a generic point. Clearly $\eta(F)$ is a finite rank vector space
over $\mathbf{C}(X)$. It induces a functor $\eta:D(X)\rightarrow
D^b(C(X)-mod)$. Complexes of sheaves with the support of
codimension $\ge 1$ are exactly the ones that become acyclic after
applying this functor. So we have a functor from the quotient
category $C(X)=D(X)/D^1(X)$ to $D^b(\mathbf{C}(X)-mod)$. All the
rank $1$ sheaves are isomorphic in $C(X)$, let us denote by $O$
it's isomorphism class. All the vector bundles on $X$ are obtained
by extensions from rank $1$ sheaves, so the triangulated
subcategory of $C(X)$ generated by cones and shifts of $O$
contains images of vector bundles and thus is equivalent to $C(X)$.\\
    If $Z$ is a closed sub-scheme of $X$ and $i:U\hookrightarrow X$ is a complement of $Z$,
then $D(X)/D_Z(X)$ is equivalent to $D(U)$, where the equivalence
is induced by a restriction functor $i^*:D(X)\rightarrow D(U)$.
It's a well-known fact. In particular if $U$ is affine then there
is a commutative algebra $Q=O_U(U)$ and $D(U)$ is equivalent to
$D^b(Q-mod)$. In particular $Hom^i_{D(X)/D_Z(X)}(O_U,O_U)=0$ for
$i\ne 0$ and $Hom^0_{D(X)/D_Z(X)}(O,O)=O_U(U)$. \\
  By definition of the quotient
category
\[Hom^i{C(X)}(A,B)=\underrightarrow{lim}_f Hom^i_{D(X)}(A,F)\]
The limit in the right hand side is taken over morphisms
$f:B\rightarrow F$ such that $Cone(f)\in D^1(X)$. For a given $f$
 let $Z_1$ be a support of $Cone(f)$. It is a closed proper
sub-scheme of $X$. We choose a closed sub-scheme $Z$, such that
$Z_1\subset Z\subset X$ and the complement of $Z$ in $X$ is
affine. So we may rewrite the formula as follows:
\[Hom^i{C(X)}(A,B)=\underrightarrow{lim}_Z\underrightarrow{lim}_g Hom^i_{D(X)}{D(A,F)}\]
First limit is over closed sub-schemes $Z$ with affine complement,
second limit is over $g$ such that $Cone(g)$ is supported in $Z$.
But now the second limit equals to
$Hom^i_{D(X)/D_Z{X}}(F,G)=Hom^i_{D(U)}(F,G)$. In particular
$Hom^i_{C(X)}(O,O)=0$ for $i\ne 0$ and
$Hom^0_{C(X)}(O,O)=\underrightarrow{lim}_Z O_U(U)=\mathbf{C}(X)$.
So the functor induced from restriction to generic point
$C(X)=D(X)/D^1(X)\rightarrow D^b(\mathbf{C}(X)-mod)$ is fully
faithful on $O$. And because every object of $C(X)$ is a direct
sum of shifts of $O$, it is fully faithful, thus an equivalence.
\end{proof}

    So we recover a field $K=\mathbf{C}(X)$ of rational functions from derived
category of coherent sheaves. Let us now recover the action of the group of birational automorphisms on it.
Suppose $\pi:X\rightarrow Y$ is a blow-up with a smooth center.\\

\begin{lemma}. The functor between quotient categories
$\Psi:C(Y)\rightarrow C(X)$ induced by $L\pi^*$ is an equivalence
and is a composition of equivalences $C(X)\leftarrow^{\eta^*_X}
D^b(K-mod)\rightarrow^{\eta^*_Y} C(Y)$.\end{lemma}
\begin{proof} If $\eta_X$ is an embedding of a generic point in
$X$ and $\eta_Y$ in $Y$, then clearly $\pi\circ\eta_X=\eta_Y$, so
in the derived categories of coherent sheaves we have a functor
isomorphism $\eta^*_X\circ L\pi^*=\eta^*_Y$

Let $E$ be an exceptional curve of this blow-up, then we have a
semi-orthogonal decomposition $D(X)=<L\pi^*D(Y),O_{|E}>$ and also
$D^1(X)=<L\pi^*D^1(Y),O_{|E}>$. So we can use that the quotient
triangulated category $<A,B>/B$ is equivalent to $A$. It follows
that $D(X)/<O_{|E}>\simeq L\pi^*D(Y)$, so
$$D(X)/D^1(X)=(D(X)/<O_{|E}>)/L\pi^*D(Y)\simeq
L\pi^*D(Y)/L\pi^*D^1(Y).$$ So $L\pi^*$ induces an equivalence
between $C(Y)$ and $C(X)$.\end{proof}

    It means that the quotient categories are independent of the
rational surface, thus implying that the Cremona group acts on
both. All line bundles placed in degree $0$ become isomorphic in
this category $C(X)$, so the Cremona group preserves the
isomorphism class of line bundles and acts in a natural way on the
endomorphism ring of this class, which is $K$. It's easy to verify
that this action coincides with the natural action of
$Cr=Aut(\mathbf{C}(x,y)/\mathbf{C})$ on $K=\mathbf{C}(x,y)$.\\

\subsection{Construction of the birational invariant}

    Here we make a general construction, which to a smooth proper
algebraic manifold $X$ associates a triangulated category
$\widetilde{C}$ invariant under birational isomorphisms.\\
    Let $\widetilde{D}(X)$ be a full subcategory of $D^b(Coh(X))$
consisting of objects left orthogonal to $O_X$, i.e. complexes of
sheaves $F^*$, such that $RHom_{D(X)}(F^*,O_X)$ is acyclic. It is
a triangulated category and it may be empty for some manifolds.\\
    Consider it's full triangulated subcategory
$\widetilde{D}^1(X)$ consisting of the complexes of sheaves that
are left orthogonal to $O_X$ and have a positive codimension of
support, it means the cohomology of the complex have support on
the subvarieties of dimension strictly less than the dimension of
$X$.\\
    Define $\widetilde{C}(X)$ as a quotient of triangulated
categories $\widetilde{C}(X)=\widetilde{D}(X)/\widetilde{D}^1(X)$
in a sense of Verdier\cite{Verdier96}.\\
    So for any object $Y$ in $\widetilde{D}(X)$ let $Q_Y$ be a
category of morphisms $f:Y\rightarrow Z$, such that $Cone(f)$ is
isomorphic to an object of $\widetilde{D}^1(X)$. Then by results
of Verdier that we reproduce from formula $(12.1)$ of
\cite{Drinfeld04} we compute the morphisms in a quotient category
this way:
    \begin{equation}  \label{morphism_quotient}
    Ext^i_{\widetilde{D}(X)/\widetilde{D}^1(X)}(U,Y)=
\rightlimit{(Y\to Z)\in Q_Y}Ext^i_{\widetilde{D}(X)}(U,Z)
    \end{equation}

\begin{theorem}\label{bir_invariant}
    If smooth $X$ and $Z$ are birationally equivalent, then
$\widetilde{C}(X)$ and $\widetilde{C}(Z)$ are equivalent as
triangulated categories.
\end{theorem}
\begin{proof}
    In virtue of the results of \cite{AKMW}(weak factorization theorem) any birational
automorphism may be decomposed as a sequence of blow-ups and
blow-downs with smooth centers. It means that it is enough to
prove our theorem for such kind of morphisms.\\
    Let $\pi:Z\rightarrow X$ be a blow-up of $X$ with a
smooth center $Y$. We have a pair of adjoint functors
$(L\pi^*,R\pi_*$ on $D(X),D(Z)$ and moreover $L\pi^*O_X=O_Z$ and
$R\pi_*O_Z=O_X$.\\
    Let us denote $D_0(Z)$ a full subcategory of $D(Z)$ consisting
of objects $F$ such that $R\pi_*F$ is acyclic, and let
$L\pi^*D(X)$ a full subcategory of $D(Z)$ consisting of images of
$L\pi^*$.\\
\begin{lemma}
$(1)$. $D_0(Z)$ and $L\pi^*D(X)$ are triangulated categories.\\
$(2)$. The natural transformation $F\rightarrow R\pi_*L\pi^*F$ in $D(X)$ is an isomorphism.\\
$(3)$. $L\pi^*$ induces an equivalence of triangulated categories
$D(X)$ and $L\pi^*D(X)$.\\
$(4)$. $L\pi^*D(X)$ is left orthogonal to $D_0(Z)$, i.e.
$RHom_{D(Z)}(L\pi^*F,G)$ is acyclic.\\
$(5)$. For any object $F\in D(X)$ there is an exact triangle
$F_1\rightarrow F\rightarrow F_0$, where $F_0\in D_0(Z), F_1\in
L\pi^*D(X)$, and moreover such a triangle is unique up to
isomorphism. $F_1$ is isomorphic to $L\pi^*R\pi_*F$.\\
$(6)$. $D_0(Z)$ is supported on exceptional divisor of the
blow-up.
\end{lemma}
\begin{proof}
$(1)$ trivial. $(2)$ is a local statement, but locally on $X$ any
object of $D(X)$ has a free bounded resolution by a structure
sheaf, for which the map $O_X\rightarrow R\pi_*L\pi^*O_X$ is an
isomorphism. $(3)$ follows from projection formula
$RHom_{D(Z)}(L\pi^*F,L\pi^*G)=RHom_{D(X)}(F,R\pi_*L\pi^*G)$ and
$(2)$. $(4)$ also follows from the projection formula. The
triangle in $(5)$ is constructed by completing a natural map
$L\pi^*R\pi_*F\rightarrow F$ to an exact triangle, map coming from
the adjunction of identity morphism $R\pi_*F\rightarrow R\pi_*F$.
Actually $F_1$ represents a functor $X\mapsto
Hom_{L\pi^*D(X)}(X,F)$ and $F_0$ represents a functor $X\mapsto
Hom_{D_0(Z)}(F,X)^*$. $(6)$ follows from the definition $R\pi_*$,
because $\pi$ is an isomorphism outside an exceptional divisor.
\end{proof}

    A concise way to summarize lemma is the statement that $D(Z)$ admits
a semi-orthogonal decomposition $D(Z)=<D_0(Z),L\pi^*D(X)>$.
Actually we have a stronger result proved at Proposition 3.4 of
\cite{Orlov92}: there is a semi-orthogonal decomposition
$D(Z)=<D(Y)_{-r+1},..,D(Y)_1,L\pi^*D(X)>$, where $r$ is a
codimension of $Y$ in $X$, but we would not use it here.\\
    We may also introduce the full triangulated subcategory $D_1(Z)\subset
D$, which consists of objects $K^{-1}_Z\otimes^L F$ for $F\in
D_0(Z)$. Because of Serre duality it is a left orthogonal to
$L\pi^*D(X)$ and we have a semi-orthogonal decomposition
$D(Z)=<L\pi^*D(X),D_1(Z)>$.\\
\begin{lemma}
    $(1)$. $D_1(Z)$ is full triangulated subcategory of
    $\widetilde{D}^1(Z)$.\\
    $(2)$. There are semi-orthogonal decompositions
    \[\widetilde{D}(Z)=<L\pi^*\widetilde{D}(X),D_1(Z)>\]
    \[\widetilde{D}^1(Z)=<L\pi^*\widetilde{D}^1(X),D_1(Z)>\]
\end{lemma}
\begin{proof}
    Objects is $D_1(Z)$ are left orthogonal to $L\pi^*D(X)$ and in
particular to $O_Z=L\pi^*O_X$. The support of an object of $D(Z)$
doesn't change when we tensor it with a linear bundle, so the
elements of $D_1(Z)$ are supported on an exceptional divisor of
blow-up, because the elements of $D_0(Z)$ are, as we see from the
last statement of the previous lemma.
\end{proof}
    This lemma implies the statement of the theorem, because to
obtain
\[\widetilde{C}(Z)=\widetilde{D}(Z)/\widetilde{D}^1(Z)=<L\pi^*\widetilde{D}(X),D_1(Z)>/<L\pi^*\widetilde{D}^1(X),D_1(Z)>\]
we may first quotient $\widetilde{D}(Z)$ by it's full subcategory
$D_1(Z)$, so we will have an equivalence
\[\widetilde{D}(Z)/\widetilde{D}^1(Z)\rightarrow L\pi^*\widetilde{D}(X)/L\pi^*\widetilde{D}^1(X)\]
And the latter is equivalent to $\widetilde{C}(X)$, because
$L\pi^*$ is fully faithful functor.
\end{proof}

\subsection{General statements about quotients of dg-categories}

    For convenience of the reader, let us summarize few facts concerning
quotients of dg-categories. Our main reference would be
\cite{Drinfeld04}.\\
    Suppose we have a dg-algebra $A^{\bullet}$ over a field $k$ with a unit.
We can see it as a dg-category $\mathcal{A}$ with one object $P$
and the endomorphism ring $Hom_{\mathcal{A}}(P,P)=A^{\bullet}$.
Then from this dg-category with one object we construct a
pre-triangulated category $\mathcal{A}^{pre-tr}$ (cf.$[D,2.4]$),
which has the formal sums $B=(\oplus_{j=1}^n P^{\oplus
m_j}[k_j],q)$ as objects. Here $q$ stands for a homogeneous degree
one matrix $q_{ij}\in Hom^1_{\mathcal{A}}(P^{\oplus m_j},P^{\oplus
m_i})[k_i-k_j]=Mat_{m_i\times m_j}(\textit{A}^{k_i-k_j+1})$. We
ask that $q_{ij}=0$ for $i\geq j$ and $dq+q^2=0$. The homomorphism
between the formal sums $B=(\oplus_{j=1}^{j=n}P^{\oplus
m_j}[k_j],q)$ and $B'=(\oplus_{j=1}^{j=n'}P^{\oplus
m'_j}[k'_j],q')$ is given by matrices $f=(f_{ij})$ where
\[f_{ij}\in Hom(P^{\oplus m_j},P^{\oplus m'_i})[k'_i-k_j]=Mat_{m'_i\times m_j}(A^{k'_i-k_j})\]
 The differential on homomorphism
groups is $df=(df_{ij})+q'f+(-1)^l fq$, where $l=deg(f_{ij})$.
Define the triangulated category $\mathcal{A}^{tr}$ as the
homotopy category of $\mathcal{A}^{pre-tr}$, which means
$Hom_{\mathcal{A}^{tr}}(X,Y)=H^0(Hom_{\mathcal{A}^{pre-tr}}(X,Y))$.
The category $\mathcal{A}^{pre-tr}$ contains cones. Let us
consider $M\in Hom_{\mathcal{A}}(P^m,P^m)$ a $m\times m$-matrix
with coefficients in $A$ of degree $0$. An element $Cone(M)$ would
be an object $(P^m\oplus P^m[1],q)$, where $q_{21}=M[-1]$ and
$q_{11}=q_{12}=q_{22}=0$. Naturally this element goes to the cone
of a morphism $M$ in the triangulated category. In the set up of
\cite{Drinfeld04} if $\mathcal{B}$ is a full dg-subcategory of
$\mathcal{A}$, then the quotient dg-category
$\mathcal{A}/\mathcal{B}$ has the same objects as $\mathcal{A}$
and for every object $X\in \mathcal{B}$ it has an additional
morphism $u_X\in Hom^{-1}_{\mathcal{A}/\mathcal{B}}(X,X)$ with
differential $d(u_X)=id_X$. The resulting homotopy category
$(\mathcal{A}/\mathcal{B})^{tr}$ is equivalent to the
quotient $\mathcal{A}^{tr}/\mathcal{B}^{tr}$.\\

\subsection{Localization in dg-sense}
    Suppose we have a closed morphism in dg-category
    $\mathcal{A}$. It will give some morphism in the homotopy category $\mathcal{A}^{tr}$ and
we want to understand what does it mean to invert this morphism.
It corresponds to taking a quotient $\mathcal{A}/Cone(a)$, so that
$(\mathcal{A}/Cone(a))^tr$ is a localization of $\mathcal{A}^{tr}$
at $a$.\\
    Let us now consider a dg-category $\mathcal{A}$ with one object $P$,
where $Hom_{\mathcal{A}}(P,P)=A^{\bullet}$ is a dg-algebra
concentrated in non-positive degrees. Let
$Q=Cone(a)\in\mathcal{A}^{pre-tr}$ for some element $a\in A^0$ of
degree $0$. We want to consider a quotient of dg-category
$D_1=<P,Q>$ which has two objects $P,Q$ by dg-subcategory
$D_2=<Q>$ with one object $Q$. Both $D_1$ and $D_2$
are seen as full subcategories of $\mathcal{A}^{pre-tr}$.\\

\begin{lemma}  Quotient dg-category
$D_1/D_2=<P,Q>/<Q>$ is equivalent to a dg-category with one object
$P$ and $End_{D_1/D_2}(P)=B$ - a dg-algebra, concentrated in
non-positive degrees. Moreover $H^0(B)$ is a localization of
$H^0(A)$ at $a$.\end{lemma}
\begin{proof} First let us write homomorphism groups in the
category $D_1$:
\begin{align*}
&End_{D_1}(P)=A^{\bullet}\\
&Hom_{D_1}(P,Q)=A^{\bullet}\oplus A^{\bullet}[1]=A^{\bullet}e_0\oplus A^{\bullet}e_{-1}\\
&Hom_{D_1}(Q,P)=A^{\bullet}\oplus A^{\bullet}[-1]=A^{\bullet}f_0\oplus A^{\bullet}f_1\\
&End_{D_1}(Q)=A^{\bullet}[1]\oplus A^{\bullet}\oplus
A^{\bullet}\oplus A^{\bullet}[-1]=A^{\bullet}e_0f_1\oplus
A^{\bullet}e_0f_0\oplus A^{\bullet}e_{-1}f_1\oplus
A^{\bullet}e_{-1}f_0
\end{align*}
    We would also have the following relations: $f_1e_0=f_0e_{-1}=0$,
$f_1e_{-1}=f_0e_0=1$, $id_Q=e_0f_0+e_{-1}f_1$. Lower indices
indicate the degrees of
homomorphisms. Also all $e_{-1}, e_0, f_0,f_1$ commute with $A^{\bullet}$.\\
    On these spaces there would be a differential:  $de_{-1}=ae_0$,
$df_0=-f_1a$.\\
    The procedure described in \cite{Drinfeld04} suggests that the
quotient category would have one more homomorphism $u\in End(Q)$
of degree $-1$, such that $du=id_Q$. So as an
$A^{\bullet}$-algebra $B=End_{D_1/D_2}(P)$ would be freely
generated by four elements: $t=f_1ue_0$, $x=f_1ue_{-1}$,
$y=f_0ue_0$, $z=-f_0ue_{-1}$. Of them $t$ would be of degree $0$,
$x,y$ - of degree $-1$, and $z$ of degree $-2$. One verifies the
following formulas: $dt=0$, $dx=ta-1$, $dy=1-at$, $dz=ax+ya$. The
elements of $B$ of degree $0$ are $A^0<z>$ - non-commutative
polynomials in $z$ with coefficients of degree $0$ in $A$, the
elements of degree $-1$ are $a_1xa_2+a_3ya_4+a_5$, where $a_i$ are
of degree $0$ and $a_5$ is of degree $-1$. So
$H^0(B)=H^0(A^{\bullet})[z]/(az-1, za-1)$ which is a localization
of $H^0(A^{\bullet})$ at $a$, and the lemma is proved.\end{proof}

\textbf{Comment}. At this point it is not clear, whether the
algebras that we obtain have non-trivial cohomologies in negative degrees.\\

    For further use we should improve this lemma for cones of more general
type. As before we consider $D_1=<P,Q>$, $D_2=<Q>$ full
subcategories of $\mathcal{A}^{pre-tr}$, but now $Q=Cone(M)$ where
$M$ is a matrix over $A^{\bullet}$ of homogeneous degree $0$. We would prove:\\

\begin{lemma} Suppose that $M=\begin{pmatrix}1_{k\times k} && 0\\
0 && a\end{pmatrix}+dM_{-1}$, $a\in A^{\bullet}$ is a degree $0$
element. Then $End_{D_1/D_2}(P)=B$ - a dg-algebra, concentrated in
non-positive degrees and $H^0(B)$ is a localization of
$H^0(A^{\bullet})$ at $a$.\end{lemma}
\begin{proof} Remember that $Cone(M)=(P^{\oplus m}\oplus P^{\oplus
m}[1], q=\begin{pmatrix}0 && M \\ 0 && 0\end{pmatrix})$. So we
would have the basis $e_0^i, e_{-1}^i$, $i=1..k+1$ of
$Hom_{D_1}(P,Q)$ as an $A^{\bullet}$-module. We would for
simplicity denote by the column vectors
$e_0=\begin{pmatrix}e_0^1\\..\\e_0^{k+1}\end{pmatrix}$,
$e_{-1}=\begin{pmatrix}e_{-1}^1\\..\\e_{-1}^{k+1}\end{pmatrix}$.
Moreover we can multiply these elements by $A^{\bullet}$ from the
left as well as from the right and these multiplications commute.
$de_{-1}=Me_0$. Then we can pick a dual basis of $Hom_{D_1}(Q,P)$:
$f_0^i,f_1^i$, $i=1..k+1$, and similarly denote row vectors
$f_0=(f_0^1,..,f_0^{k+1})$, $f_1=(f_1^1,..,f_1^{k+1})$. Then
$df_0=-f_1M$. We would have
$f_1^ie_0^j=f_0^ie_{-1}^j=0$,$f_1^ie_{-1}^j=f_0^ie_0^j=\delta_{ij}$,$id_Q=\Sigma_{i=0}^{k+1}(e_0^if_0^i+e_{-1}^if_1^i)$.\\
Again the quotient category is obtained adding element $u$ of
degree $-1$ to the space $End(Q)$ such that $du=id_Q$. So
$B=End_{D_1/D_2}(P)$ has no elements in positive degree. In degree
$0$ it is freely generated by $f_1^iue_0^j$ as $A^{0}$-algebra. So
$B^0$ as an algebra is isomorphic to $A^0<u_{ij}>$ -
non-commutative polynomials over $A^0$ and $u_{ij}$ maps through
isomorphism to $f^i_1ue^j_0$. Further notice that
$d(f_0^iue_0^j)=(df_0^i)ue_0^j+f_0^i(du)e_0^j=-(\Sigma_k
m^{ik}f_1^k ue_0^j)+f_0^iId_Qe_0^j=-(\Sigma_k dm_{-1}^{ik}(f_1^k
ue_0^j))-(m^{ii}-dm_{-1}^{ii})f_1^iue_0^j+\delta_{ij}$, where
$m_{-1}^{ij}$ stand for entries of the matrix $M_{-1}$. We know
that $d(f_1^kue_0^j)=0$, so modulo the border
$\delta_{ij}=(m^{ii}-dm_{-1}^{ii})f_1^iue_0^j$. But
$(m^{ii}-dm_{-1}^{ii})=1$ for $i\neq k+1$, so in $H^0(B)$ we have
$\delta_{ij}=f_1^iue_0^j$ for $i\neq k+1$. The similar calculation
of $d(f_0^iue_{-1}^j)$ will show that $\delta_{ij}=f_1^iue_0^j$
for $j\neq k+1$. So we deduce that $H^0(B)$ is generated by
$t=f_1^{k+1}ue_0^{k+1}$ over $H^0(A^{\bullet})$. The border will
be generated by $d(f_0^{k+1}ue_0^{k+1})=-at+1$ and
$d(f_1^{k+1}ue_{-1}^{k+1})=-ta+1$. So $H^0(B)$ is a localization
of $H^0(A^{\bullet})$ at $a$, which proves the lemma.\end{proof}

\subsection{Computation of
$\widetilde{C}(\mathbf{CP}^2)=\widetilde{D}/\widetilde{D}^1$}

    It is known that a derived category $D$ of coherent sheaves on
$\mathbf{CP}^2$ is generated by $O,O(1),O(2)$. Consider a
subcategory $\widetilde{D}$ which is left orthogonal to $O$, it is
generated by $O(1)$ and $O(2)$. Denote
$V=RHom_D(O(1),O(2))=Hom_{\mathbf{CP}^2}(O(1),O(2))$ - a
$3$-dimensional vector
space.\\

    Denote $P_1=O(1),P_2=O(2)$, then $Hom_D(P_1,P_2)=V$ and all the
other higher hom's between these objects vanish.
$End(P_1)=End(P_2)=\mathbf{C}$. Any object in $D$ may be written
as a direct sum of shifts of the complexes $P_1^m\rightarrow
P_2^n$ and the arrow is given by an $n\times m$
matrix with elements in $V$.\\

    $D^1$ is a full subcategory of $D$, consisting of complexes
of coherent sheaves with the dimension of support at most $1$.
Then $\widetilde{D}^1$ is the intersection of $D^1$ and
$\widetilde{D}$, more precisely it is  full subcategory of $D$,
consisting of objects left orthogonal to $O$ and with a dimension
of support at most $1$. The condition of left orthogonality to $O$
means that objects in $\widetilde{D}^1$ are direct sums of shifts
of $Cone(M)$, where $M:O(1)^{\oplus m}\rightarrow O(2)^{\oplus
n}$. The condition to have support of dimension at most $1$ means,
that the map given by $M$ is generically invertible, which is
exactly the case when $m=n$
and the commutative determinant of $M$ is non-trivial.\\

    Choose a non-zero element $e\in V$. First we would prove that
$\widetilde{D}/Cone(e)\simeq D^b(R-mod)$, where $D^b(R-mod)$ is
the bounded derived category of complexes $R$-modules of finite
type.\\
    We would proceed in a dg-setting, so we consider a dg-category $Q$
with two objects $P_1$, $P_2$ with $Hom(P_1,P_2)=V$ -
$3$-dimensional of degree $0$, $Hom(P_2,P_1)=0$, endomorphism
rings of both objects are one-dimensional and are generated by
identity morphisms. This category corresponds to a Kronecker
quiver with three arrows. The associated derived category
$\mathcal{Q}^{tr}$ is equivalent to $\widetilde{D}$. First we
would quotient it by an object $Cone(e:P_1\rightarrow P_2)$, where
$e$ is some non-trivial morphism. Denote by $Q_1$ a full
dg-subcategory of $\mathcal{Q}^{pre-tr}$ consisting of objects
$P_1,P_2,Cone(e)$. Let us also denote by $\mathcal{R}$ a
dg-category with one object $P$ and
$Hom_\mathcal{R}(P,P)=R=\mathbf{C}<x,y>$ - an algebra of
non-commutative polynomials on $x,y$ of degree $0$.\\
\\
\begin{lemma}Dg-categories $\mathcal{R}$ and $Q_1/Cone(e)$ are quasi-equivalent.\end{lemma}

\textbf{Corollary.} Proposition $2.5$\cite{Drinfeld04} then imply
that $\mathcal{R}^{pre-tr}$ and $(Q_1/Cone(e))^{pre-tr}$ are
quasi-equivalent, which means that $D^b(R-mod)$ and
$(Q_1/Cone(e))^{tr}$ are equivalent as triangulated categories.\\
\begin{proof}
    We understand $Cone(e)$ as a $P_2\oplus P_1[1]$ with a
differential $e:P_1\rightarrow P_2$ and symbolically write it as a
column $\begin{pmatrix}P_2\\P_1[1]\end{pmatrix}$, so the
endomorphism dg-algebra $End_{Q_1}(Cone(e))$ looks like
\[\begin{pmatrix}Hom(P_2,P_2)&&
Hom(P_1,P_2)[-1]\\Hom(P_2,P_1)[1]&&Hom(P_1,P_1)\end{pmatrix}=\begin{pmatrix}k && V[-1]\\
0 && k\end{pmatrix}\]
 $k$ means $1$-dimensional vector spaces of degree $0$, and
$V=Hom(P_1,P_2)$. The differential is the following:
$d:\begin{pmatrix}x&&v\\0&&y\end{pmatrix}\mapsto\begin{pmatrix}0&&(x-y)e\\0&&0\end{pmatrix}$.
The multiplication is coming from the matrix multiplication.
Denote this dg-algebra by $R^{\bullet}$.\\

    Choose a $2$-dimensional vector subspace $V'\subset V$, such that
$V'\oplus ke=V$ and let $a,b$ be it's basis. By construction of
\cite{Drinfeld04} dg-category $Q_1/Cone(e)$ is obtained from $Q_1$
by adding a morphism $\eta\in Hom^{-1}(Cone(e),Cone(e))$ such that
$d\eta=id_{Cone(e)}$. We also have
\[Hom_{Q_1}(P_2,Cone(e))=\begin{pmatrix}k\\0\end{pmatrix}\]
\[Hom_{Q_1}(Cone(e),P_2)=\begin{pmatrix}k&&V[-1]\end{pmatrix}\]
They are respectively left and right $R^{\bullet}$-dg-modules.
\[Hom_{Q_1/Cone(e)}(P_2,P_2)=\bigoplus_{n=0}^{\infty}
Hom_{(n)}(P_2,P_2)\]
\[Hom_{(n)}(P_2,P_2)=Hom_{Q_1}(Cone(e),P_2)\otimes\eta\otimes
R^{\bullet}\otimes\eta\otimes ... R^{\bullet}\otimes\eta\otimes
Hom_{Q_1}(P_2,Cone(e)) \]\label{hom_quotient}

$n$ is a number of times that $\eta$ shows up. In particular we
see that this is a dg-algebra concentrated in non positive
degrees. The degree $0$ part is:
\[k\oplus V\eta k\oplus V\eta V\eta k\oplus ...=T(V)\]
    $T(V)$ is a tensor algebra of $V$.\\
    This allows us to construct a DG functor $F:\mathcal{R}\rightarrow
Q_1/Cone(e)$. We define $F(P):=P_2$, and $F(x)=a\eta i$,
$F(y)=b\eta i$, where $i$ is a generator of
$Hom_{Q_1}(P_2,Cone(e))$. Remind that
$Hom_{\mathcal{R}}(P,P)=R=\mathbf{C}<x,y>=T(W)$, where $W$ is a
vector space spanned by $x,y$. So $F$ sends $T(W)$ to
$T(V')\subset T(V)=Hom_{Q_1/Cone(e)}(P_2,P_2)$.\\
    In the dg-category $Q_1/Cone(e)$ objects $P_1$ and $P_2$ are
homotopy equivalent, because $e\in Hom^0_{Q_1/Cone(e)}(P_2,P_1)$
and $j\eta i\in Hom^0_{Q_1/Cone(e)}(P_1,P_2)$ are inverse of each
other up to a homotopy($j$ is a generator of
$Hom_{Q_1}(Cone(e),P_1)$). So the functor $Ho(F)$ is essentially
surjective and to conclude the lemma we need just to prove that
$F$ induces a quasi-isomorphism between $End_{\mathcal{R}}(P)$ and
$End_{Q_1/Cone(e)}(P_2)$, in other words we have to prove that the
cohomology of the latter complex
is exactly $T(V')$.\\
    Introduce a dg-subalgebra $R^{\bullet}_0\subset R^{\bullet}$
consisting of $\begin{pmatrix}k&& ke\\0&&k\end{pmatrix}$ and let
$R'\subset R^{\bullet}$ be a submodule
$\begin{pmatrix}0&&V'\\0&&0\end{pmatrix}$, so that we have a
decomposition $R^{\bullet}=R^{\bullet}_0\oplus R'$. Introduce a
dg-subalgebra $T_{\eta}(R^{\bullet}_0)\subset
Hom_{Q_1/Cone(e)}(P_2,P_2)$ by the following formula:
\[T_{\eta}(R^{\bullet}_0)=\bigoplus_{n=0}^{\infty}\eta\otimes R^{\bullet}_0\otimes\eta...R^{\bullet}_0\otimes\eta\]
$n$ is the the number of times the $R^{\bullet}_0$ shows up.\\
The cohomology of $R^{\bullet}_0$ is one dimensional and generated
by the identity matrix $Id$. Let us introduce a subspace
$T_{\eta}(Id)\subset T_{\eta}(R^{\bullet}_0)$ defined as follows:
\[T_{\eta}(Id)=k\eta\oplus k\eta\otimes Id\otimes\eta\oplus k\eta\otimes Id\otimes\eta\otimes Id\otimes\eta\oplus...\]
    Let $C_0=\begin{pmatrix}k&& ke\end{pmatrix}$, so that we have decompositions:
\[Hom_{Q_1}(Cone(e),P_2)=C_0\oplus R'\]
\[R^{\bullet}=R^{\bullet}_0\oplus R'\]
Substituting this to the formula(\ref{hom_quotient}) we get
\[Hom_{Q_1/Cone(e)}(P_2,P_2)=\bigoplus_{m=0}^{\infty}Hom'_{m}(P_2,P_2)\]
\[Hom'_{m}(P_2,P_2)=(R'\oplus C_0)\otimes T_{\eta}(R^{\bullet}_0)\otimes R'\otimes T_{\eta}(R^{\bullet}_0)\otimes...\otimes R'\otimes T_{\eta}(R^{\bullet}_0)\otimes ki\]
Or in a more concise way
\[Hom_{Q_1/Cone(e)}(P_2,P_2)=(R'\oplus C_0)\otimes T_{\eta}(R^{\bullet}_0)\otimes T(R'\otimes T_{\eta}(R^{\bullet}_0))\otimes ki\]
Observe here that as a complex it is a direct sum of tensor
products of direct summands of complexes of the type $C_0\otimes
T_{\eta}(R^{\bullet}_0)\otimes R'$ and $R'\otimes
T_{\eta}(R^{\bullet}_0)\otimes R'$. There is a well-defined
filtration on the number of appearances of $\eta$, so we can
compute the cohomology of these complexes using the spectral
sequence with respect to these filtration. First term of the first
spectral sequence is $H^{\bullet}(C_0)\otimes
T(H^*(R^{\bullet}))\otimes R'$, $H^*$ - is cohomology. $C_0$ is
acyclic, so $H^*(C_0)=0$ and $C_0\otimes
T_{\eta}(R^{\bullet}_0)\otimes R'$ is acyclic. First term of the
second spectral sequence is $R'\otimes T_{\eta}{Id}\otimes R'$ and
the differential is given by $d(\eta)=1$. $\eta$ is odd, so the
cohomology of this complex is $R'\eta R'$. We deduce from here
that the cohomology of $Hom_{Q_1/Cone(e)}(P_2,P_2)$ is equal to
$\bigoplus_{k=0}^{\infty}R'\eta R'\eta...R'\eta ki=T(R')=T(V')$.
It's exactly what we needed to prove.

\end{proof}

    Next we have to quotient the triangulated category $D^b(R-mod)$ by
the elements $Cone(M)$, where $M:R^{\oplus k}\rightarrow R^{\oplus
k}$ is given by the matrices with entries in vector space
$V=<1,x,y>$ and non-zero commutative determinant, because this
objects generate the image of category $\widetilde{D}^1$ in the
quotient $\widetilde{D}/Cone(e)$.\\
 Let us proceed by induction on the size of matrices $M$: denote $Mat_k$ a set of $k\times k$
matrices with coefficients in a $3$-dimensional vector space
$V=<1,x,y>$ with non-negative commutative determinant. Actually we
want to invert these matrices. On each step we would obtain a
dg-algebra $A^{\bullet}_k$ concentrated in non-positive degrees,
such there the cohomology of this algebra would be the
endomorphisms of the object $P$ in a quotient category
$\mathcal{A}^{tr}/<Cone(M)|M\in Mat_k>$.\\

\begin{theorem}. There exist dg-algebras $A^{\bullet}_k$,
$k=0,1,2,..$ \label{main_thm}
concentrated in non-positive degrees, such that:\\
$1$. $A^{\bullet}_0=R=\mathbf{C}<x,y>$ in degree $0$;\\
$2$. For a triangulated category $D_k=D^b(R-mod)/{Cone(M)| M\in Mat_k}$ we have $RHom_{D_k}(P,P)\cong A^{\bullet}_k$;\\
$3$. $H^0(A_{k+1})$ is a localization of $H^0(A_k)$ at some set outside the commutator ideal;\\
$4$. All matrices in $Mat_k$ are invertible over
$H^0(A_k)$\end{theorem}

\textbf{Corollary $1$}. This procedure gives us a dg-algebra
$A^{\bullet}=\underrightarrow{\lim}A_k$, such that for the
triangulated category
$\widetilde{C}=\widetilde{D}/\widetilde{D^1}$ we have an
equivalence
$RHom_{\widetilde{C}}(P,P)\cong A^{\bullet}$.\\

\textbf{Corollary $2$}. $H=\underrightarrow{\lim}H^0(A_k)$ is a
non-commutative algebra and may be described as a consecutive
localization of $R$ at some sets, which are the elements not lying
in the commutator ideal.\\

\begin{proof} We proceed by induction: let's
put $A^{\bullet}_0=R$ concentrated in degree $0$. For $k=0$ there
is no statement to prove. Let us prove the theorem for $k+1$: take
a matrix $M\in Mat_{k+1}$, as it has non-trivial commutative
determinant, we can find sub-matrix $M'$ of size $k\times k$ that
would also have non-trivial determinant. This matrix would be
invertible over $H^0(A^{\bullet}_k)$, so let $M'_1$ be the lift of
it's inverse over $A^{\bullet}_k$. We have $M'M'_1=1+dM_2$ over
$A^{\bullet}_k$. So if matrix $M$ looks like $\begin{pmatrix}M' && b\\
a && c\end{pmatrix}$, then we would
multiply it from the left by matrix $\begin{pmatrix}M'_1 && 0\\
-aM'_1 && 1\end{pmatrix}$ and from the right by matrix
$\begin{pmatrix}1 && -M'_1 b\\ 0 && 1\end{pmatrix}$. So the
resulting matrix would be $M_0=\begin{pmatrix}1_{k\times k}&& 0 \\
0 && q\end{pmatrix}+ dM_3$, where $q=-aM'_1 b+c\in A^{\bullet}_k$
and $M_3$ is some $(k+1)\times(k+1)$ matrix of degree $-1$ defined
over $A^{\bullet}_k$.\\
    Now we can apply the previous lemma, which would say that the
quotient of $(\mathcal{A^{\bullet}}_k)^{pre-tr}$ by $Cone(M)$ is
dg-equivalent to $\mathcal{B_M}^{pre-tr}$, where $B_M$ is some
dg-algebra, concentrated in non-positive degrees and $H^0(B)$ is a
localization of $H^0(A)$ at $q$. When we quotient by a set of
matrices, the procedure of \cite{Drinfeld04} takes the direct
limit of dg-algebras, so we would obtain
$A^{\bullet}_{k+1}=\underrightarrow{\lim}_M B_M$ and
$H^0(A^{\bullet}_{k+1})$ would be just a localization of
$H^0(A^{\bullet}_k)$ at the set of elements obtained as $q$ in the
beginning of the proof. After the commutativization $q$ goes to
the determinant of the matrix $M$ divided by the determinant of
the matrix $M'$, so definitely it is non-zero, so it lies outside
the commutator ideal.\end{proof}

    As a consequence of the theorem we obtain a non-commutative
algebra
$H=H^0(A^{\bullet})=RHom^0_{\widetilde{C}(\mathbf{CP}^2)}(P,P)$
and it is a consecutive localization of $R=\mathbf{C}<x,y>$ at
some sets outside the commutator ideal. Although that is enough to
get the action of the Cremona group on the algebra $A$, because we
have an algebra morphism $H\rightarrow A$, we can prove the more
precise result:
\begin{lemma}\label{end_algebra}
    The map $H=H^{0}(A^{\bullet})\rightarrow A$ is an isomorphism.
\end{lemma}
\begin{proof}
    We can summarize the procedure of the theorem as follows:
consider a matrix $M\in M_k(V)$ with values in a vector space
$V=<1,x,y>$ and non-zero commutative determinant. Remind that $M$
is invertible over $H_k$, or more precisely up to permuting the
columns of $M$ we can represent it as a product $M=UT$, where $U$
is a lower triangular matrix with coefficients in $H_{k-1}$ and
it's diagonal elements are invertible over $H_{k-1}$ and the last
diagonal element is $1$, $T$ is an upper triangular matrix with
coefficients in $H_{k-1}$. The elements on the diagonal of $T$ are
$(1,1,..,1,\Delta)$. As proved in the theorem, $\Delta$ is
invertible in $H_k$. Let $S$ denote a subset of elements of $H$,
consisting of images of elements $\Delta$ in $H_k$ for all
matrices $M$ with non-zero commutative determinant and of any size
$k$. In particular $S$ contains non-zero elements of $V$ for
$k=1$.\\
    Let $comm:H\rightarrow K=\mathbf{C}(x,y)$ be a natural morphism
induced from the map $R=\mathbf{C}<x,y>\rightarrow K$. We will
prove that $S$ consists of all the elements of $H$ that are not
in the kernel of $comm$.\\

\begin{lemma}\label{S_is_big}
$1$. $\Delta\in S$ implies $\Delta^{-1}\in S$.\\
$2$. $\Delta_1,\Delta_2\in S$ implies $\Delta_1\Delta_2\in S$.\\
$3$. If $\Delta_1,\Delta_2\in S$ and $comm(\Delta_1+\Delta_2)\ne
0$ then $\Delta_1+\Delta_2\in S$.
\end{lemma}
\begin{proof}
    Notice first that if we have a matrix $M\in Mat_{k\times
    k}(V)$ and an equality $UM=T$, where $U$ is lower triangular
and $T$ is upper triangular with diagonal elements
$(1,..,1,\Delta)$, then for any $\lambda\in\mathbf{C}^*$ let us
consider a diagonal matrix $D=diag(1,..,1,\lambda)$. Then $MD$ has
also coefficients in $V$ and non-zero determinant and U(MD)=(TD),
where $(TD)$ has elements $(1,..,1,\lambda\Delta)$ on diagonal.
This implies that $\lambda\Delta\in S$.\\
    Given $M$ and decomposition $UM=T$, let us consider a matrix
$M'=\begin{pmatrix}M && e(k)\\ e(k)^t && 1 \end{pmatrix}$, where
$e(k)$ is a $k\times 1$ matrix with $e(k)_{i1}=0$ for all $i$
except $e(k)_{k1}=1$. Clearly $M'$ has all coefficients in $V$.
Let $\Delta_M$ be the last diagonal element of $T$ and consider
the lower-triangular matrix $U'$ such that it's non-zero entries
are $u'_{ii}=1$ for $i=1..k-1$, $u'_{kk}=\Delta^{-1}_M$,
$u'_{k+1k}=-1$, $u'_{k+1k+1}=1$. Then we have a decomposition
$(U'U)M'=T'$, here $U$ is extended to a $(k+1)\times(k+1)$ matrix
by putting $u_{k+1k+1}=1$. So we have that $T'$ has all diagonal
elements equal to $1$ except $t'_{k+1k+1}=\Delta_M^{-1}$. This
proves that if $\Delta\in S$ then $\Delta^{-1}\in S$, which is the
assertion $(1)$ of the lemma.\\
    Suppose we have $UM=T$ and $WN=Q$, and $t_{kk}=\Delta_1$,
$q_{nn}=\Delta_2$. Let $M_0$ be the first $k-1$ columns of $M$ and
$M_1$ be the last column of $M$, $N_0$ - first $(l-1)$ columns of
$N$ and $N_1$ the last column of $N$. Let $O(k,l)$ be a $k\times
l$ matrix consisting of zero's and $e(l)$ be a $l\times 1$ vector
which is zero except for $e(l)_{l1}=1$. Introduce matrices
\[P=\begin{pmatrix}M_0&& e(k)&& O(k,l-1)&&M_1\\O(l,k-1)&&N_1&&N_0&&O(l,1)\end{pmatrix}\]
\[U'=\begin{pmatrix}U&&O(k,l)\\O(l,k)&&W\end{pmatrix}\]
Then we have
\[U'P=\begin{pmatrix}UM_0&&Ue(k)&&O(k,l-1)&&UM_1\\O(l,k-1)&&WN_1&&WN_0&&O(l,1)\end{pmatrix}\]
Notice that $Ue(k)=e(k)$ because $U$ is lower triangular with the
last diagonal entry equal to $1$. Moreover the bottom entry of
$WN_1$ equals to $\Delta_2$ and the bottom entry of $UM_1$ is
$\Delta_1$. Consider now a $l\times k$ matrix
$W'=\begin{pmatrix}O(l,k-1)&& -WN_1\end{pmatrix}$ and a
$(k+l)\times(k+l)$ matrix
\[U''=\begin{pmatrix}1_{k\times k}&&O(k,l)\\W'&&1_{l\times l}\end{pmatrix}\]

$W'UM_0=0$ and the bottom entry of $W'UM_1$ is
$-\Delta_2\Delta_1$. So we see that $U''U'P$ is an
upper-triangular matrix with all the diagonal entries equal to $1$
except for the last one which is $-\Delta_2\Delta_1$. The matrix
$P$ has coefficients in vector space $V=<1,x,y>$ and $U''U'$ is
lower triangular with $1$'s on the diagonal. So it implies that
$-\Delta_2\Delta_1\in S$, which
proves the second statement of the lemma.\\

    To prove the last statement, for a $\Delta_1\in S$ we find the
decomposition $UM=T$ where $t_{kk}=-\Delta_1\in S$. Let also
$WN=Q$ with $q_{ll}=\Delta_2$. The notations for
$M_0,M_1,N_0,N_1,O(k,l),e(k)$ are as previously. Consider matrices
\[P=\begin{pmatrix}M_0&&e(k)&&O(k,l-1)&&M_1\\O(l,k-1)&&e(l)&&N_0&&N_1\end{pmatrix}\]
\[U'=\begin{pmatrix}U&&O(k,l)\\O(l,k)&&W\end{pmatrix}\]
\[U'P=\begin{pmatrix}UM_0&&e(k)&&O(k,l-1)&&UM_1\\O(l,k-1)&&e(l)&&WN_0&&WN_1\end{pmatrix}\]
The matrix $U'P$ is upper-triangular except for one entry
$(U'P)_{k+l,k}=1$. Let us multiply this matrix to the left by
$U''$ such that $u''_{ii}=1$ for all $i$, $u''_{k+l,k}=-1$ and
other entries are zero. Then $U''U'P$ would be upper triangular
with the last diagonal entry equal to $\Delta_1+\Delta_2$.
\end{proof}

    Remind that $H$ is constructed as the consecutive localization
at the elements coming from determinants. It follows from
lemma(\ref{S_is_big}) that any element outside of the kernel of
the map $comm:H\rightarrow K$ is invertible. So we invert all the
elements outside of the commutator ideal, which is a construction
of $A$. So we proved the lemma.
\end{proof}

\subsection{Action of the Cremona group on $\widetilde{C}$}

    Suppose that $f\in Cr$ is an element of the Cremona group.
There exists a diagram
\begin{align*}
\xymatrix{ X \ar[r]^g \ar[d]_{\pi_X} & Y \ar[d]^{\pi_Y} \\
\mathbf{CP}^2 \ar[r]_{f} & \mathbf{CP}^2}
\end{align*}
$\pi_X,\pi_Y$ - are sequences of blow-ups and $g$ is a regular
iso-morphism.  By Theorem\ref{bir_invariant} $(L\pi_X^*)^{-1}\circ
Lg^*\circ L\pi_Y^*$ is a composition of equivalences of
triangulated categories, so it defines an auto-equivalence
$\psi_f$ of $\widetilde{C}=\widetilde{D}/\widetilde{D}^1$.
Moreover the composition of such auto-equivalences
$\psi_f\circ\psi_g$ is non-canonically isomorphic to $\psi_{f\circ
g}$, so we have a well-defined \underline{weak} action of the
Cremona group on $\widetilde{C}$.\\

Notice that $ Lg^*\circ
L\pi_Y^*(O_{\mathbf{CP}^2}(1))=g^*(\pi_Y(O_{\mathbf{CP}^2}(1)))$
is a line bundle on $X$ left orthogonal to $O_X$.\\

\begin{lemma} The isomorphism class of an object $P=O(1)\in\widetilde{C}$ is
preserved by the action of the Cremona group.\end{lemma}
\begin{proof} As previously we suppose that our birational
automorphism is decomposed into blow-up
$\pi:X\rightarrow\mathbf{CP}^2$ and blow-down
$f:X\rightarrow\mathbf{CP}^2$ and it is given by $f\circ\pi^{-1}$.
So $L=f^*O(1)=Lf^*O(1)$ is a line bundle on $X$, left orthogonal to
$O$. If an automorphism is given by the element of
$PGL(3,\mathbf{C})$ then we don't need blowing up and $f$ is this
regular morphism, so $Lf^*O(1)=O(1)$ and the isomorphism class is
preserved. If automorphism is a quadratic transformation, then $\pi$
is blow-up of $\mathbf{CP}^2$ at three points and $E$ is an
exceptional curve. Then $Lf^*O(1)=O(2-E)$ is a line bundle and
moreover an embedding $O(2-E)\hookrightarrow O(2)$ induces
isomorphism of $Lf^*O(1)$ and $O(2)$ which is in turn isomorphic to
$O(1)$ in $\widetilde{C}(X)$.\end{proof}

\subsection{The coherence of two actions}

Cremona group acts on $\widetilde{C}$ and fixes $P$, so it acts by
outer automorphisms on $H^k(End(P))$. And as we calculated in
Lemma(\ref{end_algebra}) $H^0(End(P))=A$ - an algebra constructed
previously(\ref{construction_A}). This equality is also specified
up to inner conjugation. So we obtain an action of $Cr$ on $A$ by
outer automorphisms.\\

\begin{lemma}. The action of the Cremona group on
$A=H^0(End_{\widetilde{D}/\widetilde{D}^1}(P))$ coincides with the
one constructed by explicit
formulas(\ref{explicit_aut}).\end{lemma}
\begin{proof} Our
isomorphism between $A$ and
$H^0(End_{\widetilde{D}/\widetilde{D}^1}(P))$ depended on the choice
of the basis of vector space $H^0(\mathbf{CP}^2,O(1))=V=<e,a,b>$. We
used $e$ to identify $O(1)$ and $O(2)$ and we denoted
$e^{-1}a,e^{-1}b$ by $x,y\in End(P)$. Let $e',a',b'$ be another
basis, they are some linear combinations of $e,a,b$, so
$\gamma=e^{-1}e',\alpha=e^{-1}a',\beta=e^{-1}b'$ are some linear
combinations of $1,x,y$. So
$x'=e'^{-1}a'=(e^{-1}e')^{-1}(e^{-1}a')=\gamma^{-1}\alpha$,
$y'=\gamma^{-1}\beta$.\\
    It follows that if we have an automorphism of $\mathbf{CP}^2$ given
by $(x:y:1)\mapsto(\alpha,\beta,\gamma)$, then it would induce an
automorphism of $H^0(End_{\widetilde{D}/\widetilde{D}^1}(P))=A$
given by $(x,y)\mapsto(\gamma^{-1}\alpha,\gamma^{-1}\beta)$. We
should verify that this map is conjugate to the one, given by
explicit formulas.\\
    If $\alpha=ax+by+c$, $\beta=dy+e$, $\gamma=1$, then both maps are
written as $(x,y)\mapsto(ax+by+c,dy+e)$.\\
    If $\alpha=y$, $\beta=x$, $\gamma=1$, then both maps are given by
$(x,y)\mapsto(y,x)$.\\
    If $\alpha=x$, $\beta=1$, $\gamma=y$, then map coming from derived
category is $(x,y)\mapsto(y^{-1}x,y^{-1})$. A map defined by
explicit formulas is a composition of $(x,y)\mapsto(y^{-1}x,y)$
and $(x,y)\mapsto(x,y^{-1})$, which is
$(x,y)\mapsto(y^{-1}x,y^{-1})$, so both maps in question coincide.\\
    This three types of maps generate the action of the group
$GL(3,\mathbf{C})$ on $A$ by outer automorphisms, and we've just
proved that the two actions coincide. The Cremona group is known
to be generated by linear automorphisms and a quadratic
transformation, so we are left to verify that actions are
conjugate for quadratic
transformation $a:(x,y)\mapsto(\frac{1}{x},\frac{1}{y})$.\\
    The explicit formula have the form $a:(x,y)\mapsto(x^{-1},y^{-1})$, because $a\in D_1\times
    D_2$.\\

    As we noticed before, if $f:X\rightarrow Y$ is a blow-up at one
point, then $D(X)$ has a semi-orthogonal decomposition
$D(X)=<Lf^*D(Y),O_E>$, where $E$ is an exceptional curve and $O_E$
is it's structure sheaf.\\

    Let us now consider surface $X$ which is a blow-up of
$\mathbf{CP}^2$ at $3$ points $(0:0:1),(0:1:0),(1:0:0)$ and denote
this map by $\pi$. Let $E$ be an exceptional curve, so it is a
union of three non-intersecting rational curves. Quadratic
transformation $f:(x:y:z)\mapsto (yz,xz,xy)$ is a composition
$g\circ\pi^{-1}$, where $g$ is a regular morphism $g:X\rightarrow\mathbf{CP}^2$.\\
    $Lg^*O(1)=O(2-E)$, $Lg^*O(2)=O(4-2E)$. So the basis elements
$e,a,b\in Hom(O(1),O(2))$ maps by $Lg^*$ to some elements
$e',a',b'\in Hom_X(O(2-E),O(4-2E))$, so to understand the action
of this quadratic transformation on non-commutative algebra $A$ we
need to understand the elements $e'^{-1}a',e'^{-1}b'\in
End_{\widetilde{D}/\widetilde{D}^1}(P)$.\\

    Consider a sheaf $O(1-E)$ on $X$. Let us prove that it is left
orthogonal to $O_X$, $RHom(O(1-E),O)=R\Gamma(O(E-1))$ and we have
an exact sequence of sheaves $0\rightarrow O(-1)\rightarrow
O(-1+E)\rightarrow O(-1+E)_{|E}\rightarrow 0$. Actually $O(-1+E)$
restricted to any component $E_i$ of $E$ is isomorphic to
$O(E_i)_{|E_i}\cong O(-1)_{|\mathbf{CP}^1}$, so it's acyclic.
$O(-1)_X$ is also acyclic, so $R\Gamma(O(-1+E))=0$ which proves that $O(1-E)\in\widetilde{D}(X)$.\\
    Now in $\widetilde{D}(X)$ we have the following maps:\\
\begin{align*}
&u_x,u_y,u_z\in Hom(O(1-E),O(2-E))=Hom(O(1),O(2))\\
&f=xyz\in Hom(O(1),O(4-2E))\\
&a',b',e'\in Hom(O(2-E),O(4-2E))\\
&i:O(1-E)\hookrightarrow O(1), j:O(2-E)\hookrightarrow O(2)
\end{align*}
    By definition: $j\circ u_x=x\circ i$, $j\circ u_y=y\circ i$, $j\circ
u_z=z\circ i$. Also we have the following relations: $f\circ
i=a'\circ u_x=b'\circ u_y=e'\circ u_z$, just because $a'=yz, b'=xz,
c'=xy$ as elements of $Hom(O(2-E),O(4-2E))=\Gamma(O(2-E))$. \\
    Now we can pass to a quotient category
$\widetilde{C}(X)=\widetilde{D}(X)/\widetilde{D}^1(X)$, where all
the previously described maps become invertible, because they
define non-trivial maps between line bundles, so their cones have
support of dimension $1$. So in the quotient category we have:
$a'\circ j^{-1}\circ x\circ i=e'\circ j^{-1}\circ z\circ i$ which
implies $e'^{-1}a'=j^{-1}\circ zx^{-1}\circ j$ and
$e'^{-1}b'=j^{-1}\circ zy^{-1}\circ j$, so this map is conjugate
to the map $(x,y)\mapsto (x^{-1},y^{-1})$, which concludes the
lemma.\end{proof}

\bibliographystyle{plain}
\bibliography{references}

\end{document}